\documentclass[12pt,reqno]{amsart}

\usepackage{amssymb}

\newcommand{\alp}{\alpha}
\newcommand{\sig}{\sigma}

\newcommand{\Sig}{\Sigma}

\newcommand{\longc}{,\dotsc,}
\newcommand{\longp}{+\dotsb+}
\newcommand{\longe}{=\dotsb=}

\newcommand{\seq}{\subseteq}
\newcommand{\stm}{\setminus}
\newcommand{\est}{\varnothing}

\renewcommand{\>}{\rangle}

\newcommand{\lfl}{\left\lfloor}
\newcommand{\rfl}{\right\rfloor}

\newtheorem{lemma}{Lemma}
\newtheorem{theorem}{Theorem}
\newtheorem{corollary}{Corollary}

\theoremstyle{remark}
\newtheorem{remark}{Remark}

\newcommand{\refc}[1]{\ref{c:#1}}
\newcommand{\refl}[1]{\ref{l:#1}}
\newcommand{\reft}[1]{\ref{t:#1}}
\newcommand{\refs}[1]{\ref{s:#1}}
\newcommand{\refe}[1]{\eqref{e:#1}}
\newcommand{\refb}[1]{\cite{b:#1}}

\DeclareMathOperator{\ord}{ord}
\DeclareMathOperator{\supp}{supp}

\newcommand{\Su}[1][\alp]{\Sig(#1)}
\newcommand{\Suc}[1][\alp]{\Sig^{\scriptscriptstyle{\times}}(#1)}

\newcommand{\h}{{\mathsf H}}

\title[The number of subsequence sums]%
  {A nonlinear bound \\ for the number of subsequence sums}
\author{Vsevolod F. Lev}
\email{seva@math.haifa.ac.il}
\address{Department of Mathematics, The University of Haifa at Oranim,
  Tivon 36006, Israel}
\makeatletter
\@namedef{subjclassname@2020}{\textup{2020} Mathematics Subject Classification}
\makeatother

\begin{document}
\baselineskip=16pt

\begin{abstract}
We show that a finite zero-sum-free sequence $\alp$ over an abelian group
has at least $c|\alp|^{4/3}$ distinct subsequence sums, unless $\alp$ is
``controlled'' by a small number of its terms; here $|\alp|$ denotes the
number of terms of $\alp$, and $c>0$ is an absolute constant.
\end{abstract}

\maketitle

\section{Background and motivation}

A classical object of study of the combinatorial number theory is the
Davenport constant of a finite abelian group, defined to be the smallest
integer $n$ such that every $n$-term sequence of group elements is guaranteed
to have a nonempty subsequence with the zero sum of its terms. According to
Olson~\refb{o1}, the problem of finding this group invariant was raised by
Davenport at the Midwestern Conference on Group Theory and Number Theory held
in the Ohio State University in April 1966, in connection with the
investigation of the class groups of the algebraic number fields. However,
this exact problem has been considered three years earlier by Rogers~\refb{r}
who, in turn, refers to a personal communication with C.~Sudler. Other early
papers where this or related problems are studied include
\cite{b:vebk,b:b,b:veb,b:o2}.

The precise value of the Davenport constant is still unknown in general,
despite a large number of partial results. Eggleton and Erd\H os
\cite[Theorem~3]{b:ee} have shown that if the underlying group $G$ is not
cyclic, then its Davenport constant does not exceed $(|G|+1)/2$; that is, if
$\alp$ is a finite sequence free of nonempty zero-sum subsequences, then its
length satisfies $|\alp|\le |G|/2$. This was improved by Olson and
White~\refb{ow} who proved that if $\alp$ is free of nonempty zero-sum
subsequences, then it has at least $2|\alp|$ distinct subsequence sums
provided that the subgroup generated by the elements of $\alp$ is not cyclic;
here the empty subsequence with the zero sum is counted, too. Indeed, the
result of Olson-White applies to infinite and nonabelian groups as well.

Evidently, the reason for Eggleton-Erd\H os and Olson-White to ignore the
cyclic groups is that for these groups the Davenport constant is easily seen
to be equal to the order of the group. This simple fact suggests,
nevertheless, an interesting research avenue to explore: what is the
structure of ``long'' sequences free of nonempty zero-sum subsequences in the
finite cyclic groups?

We introduce some basic terminology to proceed.

A finite sequence of elements of an abelian group is called
\emph{zero-sum-free} if all of its nonempty subsequences have a nonzero sum
of their terms; the sequence is \emph{minimal zero-sum} if the sum of all of
its terms is zero, while all of its nonempty proper subsequences have nonzero
term sum. The two classes of sequences are actually quite close to each other
due to the observation that removing any term from a minimal zero-sum
sequence results in a zero-sum-free sequence, and conversely, any
zero-sum-free sequence can be turned into a minimal zero-sum sequence by
appending an appropriate group element to it.

Suppose that the underlying group $G$ is cyclic of finite order. Addressing
the above-mentioned problem, Gao~\refb{g} characterized minimal zero-sum
sequences $\alp$ of length $n:=|\alp|>2|G|/3+O(1)$; specifically, for any
such sequence there is a generating element $g\in G$ and positive integers
$x_1\longc x_n$ with $x_1\longp x_n=|G|$ such that the $n$ terms of $\alp$
are $x_1g\longc x_ng$. Savchev and Chen~\refb{sc2} and, independently,
Yuan~\refb{y} have shown that, indeed, Gao's characterization stays true for
all zero-sum-free sequences of length $|\alp|>|G|/2$. In~\refb{l}, these
results are extended the same way the theorem of Olson-White extends the
result of Eggleton-Erd\H os; namely, \cite[Theorem~2]{b:l} shows that the
assumption $|\alp|>|G|/2$ can be relaxed to $|\alp|>|\Su|/2$, where $\Su$ is
the set of all subsequence sums of $\alp$, including the empty subsequence.

A significant further progress was made by Savchev and Chen~\refb{sc3} who
have introduced a subtle and involved argument to classify minimal zero-sum
sequences $\alp$ in a finite cyclic group $G$, given that $|\alp|\ge\lfl
|G|/3\rfl+3$. As shown in~\refb{sc3}, under this assumption there exist group
elements $u\ne 0$ and $v$ satisfying a number of conditions such that all
terms or $\alp$ are contained either in the subgroup $\<u\>$ generated by
$u$, or in the coset $v+\<u\>$. We omit the (somewhat technical) exact
statement of this result referring the interested reader to~\refb{sc3}
instead.

\section{The main result}\label{s:main}

In this section we use the notation which will not be introduced formally
until Section~\refs{notation}. We believe, however, that our notation is
standard enough to not cause any discomfort to the reader. Nevertheless, it
may be worth remarking that $\Su$ denotes the set of \emph{all} subsequence
sums of the sequence $\alp$, including the empty subsequence, while $\Suc$ is
defined the same way, except that the empty subsequence is ignored; say, if
$\alp$ is the sequence of integers containing one term equal to $-3$ and two
terms equal to $2$, then $\Su=\{-3,-1,0,1,2,4\}$, while
$\Suc=\{-3,-1,1,2,4\}$. By $\supp(\alp)$ we denote the set of all terms of
$\alp$, and by $|\alp|$ the length (that is, the number of terms) of $\alp$.

The aim of the present paper is to go beyond the Savchev-Chen bound
 $|\alp|\ge\lfl |G|/3\rfl+3$.
\begin{theorem}\label{t:main}
Suppose that $C\ge 2$ is an integer, and $\alp$ is a zero-sum-free sequence
of length $n:=|\alp|\ge(4C)^3$ over an abelian group. If
$|\Su|<Cn-C^2\sqrt{6n}$, then there are an integer $h>n-C\sqrt{6n}$ and an
element $a\in\supp(\alp)$ of order $\ord(a)>h+1$ such that, letting
$P:=\{a,2a,\longc ha\}$, we have $\Su[\supp(\alp)]+P\seq\Suc$.

Moreover, there is a set $X\seq\Su[\supp(\alp)]$ with $0\in X$ and $|X|\le
C-1$ such that $\Su[\supp(\alp)]\seq X+P-P$; in particular, each term of
$\alp$ can be written as $x+ta$ with an element $x\in X$ and an integer
$t\in[-(h-1),h-1]$.
%
\end{theorem}

\begin{remark}
The arithmetic progression $P-P=\{-(h-1)a\longc(h-1)a\}$ can be considered as
a ``local approximation'' to the subgroup $\<a\>$ generated by $a$. The
assertion of the theorem remains true, but becomes weaker with $P-P$ replaced
by $\<a\>$.
\end{remark}

\begin{remark}
The translates $x+P-P$ with $x\in X$ are not guaranteed to be pairwise
disjoint, but it is immediate from the proof that the translates $x+P$ are
disjoint.
\end{remark}

\begin{remark}
The bound $|X|\le C-1$ is best possible. To see this, fix a positive integer
$s$, set $C:=2^s+1$, choose linearly independent vectors $e_0,e_1\longc e_s$
in a vector space over a field of characteristic $0$, and consider the
sequence $\alp=e_0^{n-s}e_1\dotsb e_s$. We have
  $$ \Su=\{x_1e_1\longp x_se_s+te_0\colon x_1\longc x_s\in\{0,1\},
                                                   \ 0\le t\le n-s \} $$
whence
  $$ |\Su| = 2^s(n-s+1)=(C-1)n-(C-1)(s-1) < Cn - C^2\sqrt{6n} $$
provided that $n$ is sufficiently large. On the other hand, one cannot find
an arithmetic progression $P$ and a set $X$ of size $|X|\le C-2=2^s-1$ such
that $\Su[\supp(\alp)]\seq X+P$ since the right-hand side is contained in a
system of $|X|$ lines, each of them covering at most two vertices of the
$(s+1)$-dimensional cube
 $\Su[e_0e_1\dotsb e_s]=\Su[\supp(\alp)]$.
\end{remark}

\begin{remark}
The assumptions $n\ge(4C)^3$ and $|\Su|<Cn-C^2\sqrt{6n}$ essentially mean
that our result applies to zero-sum-free sequences satisfying
$|\Su|<\frac14|\alp|^{4/3}$. Conjecturally, the former assumption can be
dropped or at least relaxed very significantly, and the latter can be
replaced with the weaker $|\Su|\le Cn-(C-1)^2$. If true, this is best
possible as it follows by considering the sequence
$\alp:=g_1^{n-C+1}g_2^{C-1}$, where $g_1$ and $g_2$ are independent group
elements; in this example (originating from~\cite[Examples~1 and~2]{b:l0}) we
have $|\Su|=Cn-(C-1)^2+1$, while $\Su$ is not a union of $C-1$ or less
arithmetic progressions with the same difference.

Some minor adjustments of the constants can be obtained by fine-tuning our
calculations, but we prefer instead to keep the argument clean and reasonably
simple.
\end{remark}

\begin{remark}
Theorem~\reft{main} markedly lacks the precision of the result of Savchev and
Chen, the latter establishing a necessary and sufficient condition for $\alp$
to be minimal zero-sum under the assumption
 $|\alp|\ge\lfl|G|/3\rfl+3$, where $G$ is the underlying group. On the other
hand, Theorem~\reft{main} is not restricted to the coefficient $C=3$,
allowing $C$ to be \emph{any} positive integer. Also, the assumption
$|\Su|<C(1+o(1))|\alp|$ is substantially weaker than the assumption
$|\alp|>(C^{-1}+o(1))|G|$. Finally, our result applies to any abelian group,
not necessarily finite or cyclic. As explained in~\refb{sc3}, the extension
from the finite cyclic to arbitrary abelian groups is meaningless in the
situation where $|\alp|\ge\lfl |G|/3\rfl+3$; however, it makes perfect sense
in the settings of Theorem~\reft{main}.
\end{remark}

We now briefly outline the idea behind the proof.

The crucial role is played by the transfer operation introduced in
Section~\refs{stab}. The operation modifies a given sequence $\alp$ so that,
in particular, the subsequence sum set $\Su$ does not increase, while the
length $|\alp|$ either increases, or stays the same, but in the latter case
the sequence $\alp$ becomes, loosely speaking, ``more concentrated''.
Moreover, the support $\supp(\alp)$ stays unchanged, and the property of
being zero-sum-free is preserved. Having applied the transfer operation
sufficiently many times, we finally reach a zero-sum-free sequence which is
stable under further applications of the operation. This resulting stable
sequence turns out to contain a term of very large multiplicity, making it
easier to analyze than a ``generic'' zero-sum-free sequence, and its relation
to the original sequence $\alp$ is strong enough to read off the properties
of the latter from those of the former.

We introduce the notation used in the next section. Three basic results
needed for the proof of Theorem~\reft{main} are quoted in Section~\refs{aux}.
In Section~\refs{stab} we define the transfer operation and study its
properties. Theorem~\reft{main} is proved in the concluding
Section~\refs{proof}.

\section{Notation}\label{s:notation}

From now on, by a \emph{sequence} we mean a finite sequence of elements of an
abelian group. Informally, a sequence is an unordered list of the group
elements, with repetitions allowed. Formally, sequences are elements of the
abelian monoid freely generated by the elements of the group. We use
multiplicative notation for the monoid operation; thus, for instance, the
``concatenation'' of the sequences $\alp_1\longc\alp_k$ is the product
sequence $\alp_1\dotsb\alp_k$.

Up to the order of the factors, every sequence $\alp$ can be uniquely written
in the form $\alp=a_1^{m_1}\dotsb a_s^{m_s}$ where $s\ge 0$ is an integer,
$a_1\longc a_s$ are pairwise distinct group elements, and $m_1\longc m_s$ are
positive integers. The elements $a_i$ are called the \emph{terms} of $\alp$,
and the integers $m_i$ are the \emph{multiplicities} of the corresponding
terms. Alternatively, the terms of $\alp$ are group elements dividing $\alp$,
and the multiplicity of a group element $a$ in $\alp$, denoted $\nu_\alp(a)$
below, is the largest integer $m\ge0$ such that $a^m\mid\alp$.

For a nonempty sequence $\alp$, by $\h(\alp)$ we denote the largest
multiplicity of a term of $\alp$: that is,
$\h(\alp)=\max\{\nu_{\alp}(a)\colon a\mid\alp\}$. More generally, by
$\h_k(\alp)$ we denote the $k$th largest multiplicity of a term of $\alp$;
thus, $\h(\alp)=\h_1(\alp)$. By convention, $\h_k(\est)=0$ for all $k\ge 1$.

The \emph{length} of $\alp$ is defined by $|\alp|:=m_1\longp m_s$. We also
let $|\alp|_2:={m_1}^2\longp m_s^2$.

The set $\{a_1\longc a_s\}$ is called the \emph{support} of $\alp$ and
denoted $\supp(\alp)$; therefore, $s=|\supp(\alp)|$, and for a group element
$a$, we have $a\mid\alp$ if and only if $a\in\supp(\alp)$, meaning that $a$
is a term of $\alp$.

\emph{Subsequences} of $\alp$ are sequences $\alp'$ with $\alp'\mid\alp$.
Writing $\alp=a_1^{m_1}\dotsb a_s^{m_s}$ with $s,a_i$, and $m_i$ as above,
$\alp'$ is a subsequence of $\alp$ if and only if $\alp'=a_1^{m_1'}\dotsb
a_s^{m_s'}$ where $0\le m_i'\le m_i$ for all $i=1\longc s$.

We denote by $\sig(\alp)$ the sum of all terms of $\alp$ with multiplicities
counted: $\sig(\alp) = m_1a_1\longp m_sa_s$.

The subsequence sum set of $\alp$ is the set
  $$ \Su := \{ \sig(\alp') \colon \alp'\mid\alp \}. $$
We also write
  $$ \Suc := \{ \sig(\alp') \colon \alp'\mid\alp,\ \alp'\ne\est \}; $$
thus, $\Suc[\est]=\est$ and $\Su=\Suc\cup\{0\}$. The sequence $\alp$ is
\emph{zero-sum-free} if it does not have a nonempty zero-sum subsequence;
that is, if $0\notin\Suc$.

For subsets $A_1\longc A_k$ of an abelian group, we let
  $$ A_1\longp A_k
                := \{ a_1\longp a_k \colon a_1\in A_1\longc a_k\in A_k \} $$
and
  $$ A_i-A_j := \{ a_i-a_j \colon a_i\in A_i,\ a_j\in A_j \}. $$

The subgroup generated by a group element $g$ is denoted by $\<g\>$, and the
order of $g$ by $\ord(g)$.

\section{Basic Facts}\label{s:aux}

We need the following, well-known, lower-bound estimates for the size of the
subsequence sum set of a zero-sum-free sequence.

\begin{lemma}[{\cite[Theorem 2.5]{b:sf}};
  see also {\cite[Proposition 5.3.5]{b:gh}}]\label{l:sf}
If $\alp$ is a finite zero-sum-free sequence over an abelian group, then
  $$ |\Su| \ge 2|\alp| - \h(\alp) + 1 \ge |\alp| + |\supp(\alp)|. $$
\end{lemma}

\begin{lemma}[{\cite[Proof of Lemma~1]{b:ben}};
  see also {\cite[Theorem~5.3.1]{b:gh}}]\label{l:ben}
Suppose that $k$ is a positive integer, and $\alp=\alp_1\dotsb\alp_k$ is a
factorization of a finite, zero-sum-free sequence $\alp$ over an abelian
group. Then
  $$ |\Su| \ge |\Su[\alp_1]|\longp|\Su[\alp_k]|-(k-1). $$
\end{lemma}

\begin{lemma}[{\cite[Theorem~1.1]{b:hui}}]\label{l:hui}
If A is a finite, zero-sum-free \emph{subset} of an abelian group (so that
$\h(A)=1$), then $|\Su[A]|\ge 1+\frac16\,|A|^2$.
\end{lemma}

\section{The transfer operation}\label{s:stab}

A simple construction presented in Section~\refs{proof} after the statement
of Lemma~\refl{bounds} shows that a zero-sum-free sequence $\alp$ with about
$C|\alp|$ subsequence sums may fail to contain a term of multiplicity larger
than, roughly, $|\alp|/C$. The machinery of transfer operation developed in
this section supplies a sort of a poor man's compensation for the lack of
high-multiplicity terms.

Throughout this section, $\alp$ is a finite sequence of elements of an
abelian group of length $n:=|\alp|$, written as $\alp=a_1^{m_1}\dotsc
a_s^{m_s}$ where $s\ge 2$, the terms $a_1\longc a_s$ are pairwise distinct,
and the exponents $m_1\longc m_s$ are positive integers.

\begin{lemma}\label{l:prec}
Suppose that $i,j\in[1,s],\ i\ne j$, and $v<m_j$ is a positive integer.
Consider the sequence $\alp'=a_1^{m_1'}\dotsb\alp_s^{m_s'}$ defined by
$m'_i=m_i+v$, $m_j'=m_j-v$, and $m_k'=m_k$ for $k\in[1,s]\stm\{i,j\}$. Then
$|\alp'|_2>|\alp|_2$ if and only if $v>m_j-m_i$.
\end{lemma}

\begin{proof}
It suffices to notice that the inequality $(m_i+v)^2+(m_j-v)^2>m_i^2+m_j^2$
reduces to $v>m_j-m_i$.
\end{proof}

Fix indices $i,j\in[1,s]$ with $i\ne j$ and write $\alp=a_i^{m_i}
a_j^{m_j}\rho$ where $\rho$ is a sequence with $a_i,a_j\notin\supp(\rho)$. If
there exist positive integers $u\le m_i+1$ and $v\le m_j-1$ such that
$ua_i=va_j$, and either $u>v$, or $u=v>m_j-m_i$, then we say that the
sequence $a_i^{m_i+u}a_j^{m_j-v}\rho$ is obtained from the original sequence
$\alp$ by the transfer operation. We also say in this case that the (ordered)
pair $(i,j)$ is \emph{unstable}. If $u$ and $v$ with the specified properties
do not exist, then we say that $(i,j)$ is \emph{stable}. If both $(i,j)$ and
$(j,i)$ are stable pairs, then we say that the set $\{i,j\}$ is stable.
Finally, we say that the original sequence $\alp$ is stable if all ordered
pairs $(i,j)$ are stable; equivalently, if all two-element subsets
$\{i,j\}\seq[1,s]$ are stable.

We remark that the inequality $v\le m_j-1$ ensures that $\supp(\alp)$ is not
affected by the transfer operation, while the condition $u\le m_i+1$ is
needed to guarantee that the set $\Suc$ of subsequence sums does not
increase; see Lemma~\refl{swapprop} below.

\begin{lemma}\label{l:ijji}
Suppose that $i,j\in[1,s],\ i\ne j$.
\begin{itemize}
\item[(i)] If $m_i\ge m_j$, then for the pair $(i,j)$ to be unstable it is
    necessary and sufficient that there exist positive integers $u\le
    m_i+1$ and $v\le m_j-1$ such that $ua_i=va_j$ and $u\ge v$.
\item[(ii)] For the pair $(j,i)$ to be unstable it is necessary and
    sufficient that there exist positive integers $u\le m_i-1$ and $v\le
    m_j+1$ such that $ua_i=va_j$ and either $u<v$, or $u=v>m_i-m_j$.
\end{itemize}
\end{lemma}

\begin{proof}
The first assertion follows directly from the definition of a stable pair.
The second assertion is the definition of an unstable pair in disguise, which
is immediately seen by interchanging $i$ and $j$ and, simultaneously, $u$ and
$v$.
\end{proof}

\begin{corollary}\label{c:setstab}
Suppose that $i,j\in[1,s],\ i\ne j$, and $m_i\ge m_j$. If there exist
positive integers $u\le m_i+1$ and $v\le m_j-1$ such that $ua_i=va_j$, then
$\{i,j\}$ is unstable.
\end{corollary}

\begin{proof}
Suppose that $u\le m_i+1$ and $v\le m_j-1$ satisfy $ua_i=va_j$. If $u\ge v$,
then $(i,j)$ is unstable by Lemma~\refl{ijji} (i). If $u<v$ then $u< m_j-1\le
m_i-1$ and $(j,i)$ is unstable by Lemma~\refl{ijji} (ii).
\end{proof}

For $j\in[1,s]$, we write $P_j:=\{0,a_j,2a_j\longc m_ja_j\}$ and
$P_j':=\{0,a_j,2a_j\longc(m_j-1)a_j\}$. Notice that if $\alp$ is
zero-sum-free, then $|P_j|=m_j$.
\begin{corollary}\label{c:zeroint}
Suppose that $i,j\in[1,s],\ i\ne j$, and $m_i\ge m_j$. If $\{i,j\}$ is
stable, then $P_i\cap P_j'=\{0\}$.
\end{corollary}

\begin{proof}
If $P_i$ and $P_j'$ share a common element $g\ne0$, then there exist positive
integers $u\le m_i$ and $v\le m_j-1$ such that $g=ua_i$ and $g=va_j$.
Therefore $ua_i=va_j$ and $\{i,j\}$ is unstable by Corollary~\refc{setstab}.
\end{proof}

\begin{corollary}\label{c:PiPj}
Suppose that $\alp$ is zero-sum-free, and that $i,j\in[1,s],\ i\ne j$, and
$m_i\ge m_j$. If $\{i,j\}$ is stable, then $|P_i+P_j'|=(m_i+1)m_j$ and,
consequently, $|\Su|\ge(m_i+1)m_j$.
\end{corollary}

\begin{proof}
It suffices to show that all sums $ua_i+va_j$ with $u\in[0,m_i]$,
$v\in[0,m_j-1]$ are pairwise distinct. To this end, suppose that
$u'a_i+v'a_j=u''a_i+v''a_j$ with $u',u''\in[0,m_i]$ and $v',v''\in[0,m_j-1]$.
Then, letting $u:=u''-u'$ and $v:=v'-v''$, we have $ua_i=va_j$ where $|u|\le
m_i$ and $|v|\le m_j-1$. If $u$ and $v$ are distinct from $0$, then they are
of the same sign since $\alp$ is zero-sum-free; without loss of generality,
both are positive, and then the equality $ua_i=va_j$ contradicts
Corollary~\refc{setstab}.
\end{proof}

\begin{corollary}\label{c:TBL}
Suppose that $\alp$ is zero-sum-free, and that $i,j\in[1,s]$, $i\ne j$,
$m_i=\h_1(\alp)$, and $m_j=\h_2(\alp)$. If $\{i,j\}$ is stable, then
$|\Su|-2n\ge(m_i-2)(m_j-2)-2$.
\end{corollary}

\begin{proof}
Let $\rho$ be the sequence defined by $\alp=a_i^{m_i}a_j^{m_j-1}\rho$. By
Corollary~\refc{PiPj} and Lemmas~\refl{ben} and~\refl{sf}, and in view of
$\h(\rho)\le m_j$,
\begin{align*}
 |\Su| &\ge |\Su[a_i^{m_i}a_j^{m_j-1}]| + |\Su[\rho]| - 1 \\
       &\ge |P_i+P_j'| + (2(n-m_i-(m_j-1)) - m_j+1) - 1 \\
       &=   (m_i+1)m_j + 2n-2m_i-3m_j+ 2 \\
       &=    m_im_j - 2m_i - 2m_j + 2 + 2n \\
       &=   (m_i-2)(m_j-2) - 2 + 2n.
\end{align*}
\end{proof}

\begin{lemma}\label{l:swapprop}
If $\alp'$ is obtained from $\alp$ by a series of subsequent transfer
operations, then
  $$ \Suc[\alp']\seq\Suc,\ \supp(\alp')=\supp(\alp), $$
and
  $$ \text{either $|\alp'|>|\alp|$, or $|\alp'|=|\alp|$ and
                                           $|\alp|_2<|\alp'|_2$.} $$
\end{lemma}

\begin{proof}
Without loss of generality, we assume that $\alp'$ is obtained from $\alp$ by
just one transfer operation; say, $\alp=a_i^{m_i}a_j^{m_j}\rho$ and
$\alp'=a_i^{m_i+u}a_j^{m_j-v}\rho$ where $i,j\in[1,s]$, $i\ne j$,
$a_i,a_j\notin\supp(\rho)$, $u\in[1,m_i+1]$, $v\in[1,m_j-1]$, $ua_i=va_j$,
and either $u>v$, or $u=v>m_j-m_i$. We prove the inclusion
$\Suc[\alp']\seq\Suc$; the rest follows in a straightforward way from
Lemma~\refl{prec}.

Comparing the identities
\begin{align*}
  \Suc[a_i^{m_i}a_j^{m_j}\rho]
     &= (\Suc[a_i^{m_i}a_j^{m_j}]+\Su[\rho])
             \cup\Suc[\rho] \\
\intertext{and}
  \Suc[a_i^{m_i+u}a_j^{m_j-v}\rho]
     &= (\Suc[a_i^{m_i+u}a_j^{m_j-v})+\Su[\rho]] \cup \Suc[\rho],
\end{align*}
it suffices to show that
$\Suc[a_i^{m_i+u}a_j^{m_j-v}]\seq\Suc[a_i^{m_i}a_j^{m_j}]$.

Suppose thus that $g\in\Suc[a_i^{m_i+u}a_j^{m_j-v}]$ and write $g=xa_i+ya_j$
where $x\in[0,m_i+u]$ and $y\in[0,m_j-v]$ are not both equal to $0$. We show
that $g\in\Suc[a_i^{m_i}a_j^{m_j}]$. Indeed, the case $x\le m_i$ is trivial,
while if $x\ge m_i+1$, then $x-u\in[0,m_i]$ and $y+v\in[1,m_j]$;
consequently, $g=(x-u)a_i+(y+v)a_j\in\Suc[a_i^{m_i}a_j^{m_j}]$, as wanted.
\end{proof}

If $\alp$ and $\alp'$ are as in Lemma~\refl{swapprop} and $\alp$ is
zero-sum-free, then so is $\alp'$ in view of $\Suc[\alp']\seq\Suc$. On the
other hand, from $|\alp'|\le|\Su[\alp']|\le|\Su|$ we conclude that the
possible lengths $|\alp'|$ are uniformly bounded; therefore, in view of the
second assertion of the lemma, having applied a series of subsequent transfer
operations to any given zero-sum-free sequence, we eventually obtain a stable
zero-sum-free sequence.

\section{Proof of Theorem~\reft{main}}\label{s:proof}

We start with a lemma establishing a number of relations between the length
and the largest multiplicities of the terms of a zero-sum-free sequence
$\alp$ with a small number of subset sums; most importantly, the lemma shows
that if $\alp$ is stable, then it has a term of very large multiplicity,
while all other terms have small multiplicities.

\begin{lemma}\label{l:bounds}
Suppose that $C\ge 2$ is an integer, and $\alp$ is a zero-sum-free sequence
over an abelian group with $n:=|\alp|\ge 2$. Let
$h_1:=\h_1(\alp)$ and $h_2:=\h_2(\alp)$. If $|\Su|<Cn$, then
\begin{align}
  h_1 &> n^2/(6|\Su|) > n/(6C) \label{e:lbii}
\intertext{and}
  h_1 &> n-\sqrt{6h_2|\Su|}+h_2 > n-\sqrt{6Cnh_2}+h_2. \label{e:lbiii}
\intertext{If, moreover, $\alp$ is stable, $n\ge(4C)^3$, and
           $|\Su|<Cn-C^2\sqrt{6n}$, then}
  h_2 &\le C-1 \label{e:newb}
  \intertext{and}
  h_1 &> n-C\sqrt{6n}+1. \label{e:bounds}
\end{align}
\end{lemma}

We remark that only the last estimate of the lemma is used outside of the
lemma itself; however, we believe that the first three estimates can be
useful elsewhere, and for this reason we have included them into the
statement.

Interestingly, the bound~\refe{lbii} is best possible save for the
coefficient $1/6$. To see this, fix integer numbers $k,s\ge 1$ and consider
the sequence $\alp:=(g\cdot2g\dotsb sg)^k$ where $g$ is a group element of
sufficiently large order. We have $n:=|\alp|=ks$, $\h(\alp)=k$, and
  $$ |\Su| = 1 + \frac12s(s+1)\,k = \frac12(s+1)n + 1. $$
Choosing $s:=2(C-1)$ we therefore get $|\Su|<Cn$ while
$\h(\alp)=n/s=n/(2C-2)$.

\begin{proof}[Proof of Lemma~\refl{bounds}]
For an integer $k\in[1,h_1]$, let $A_k$ be the set of all terms of $\alp$ of
multiplicity at least $k$:
  $$ A_k := \{a\in\supp(\alp)\colon \nu_\alp(a)\ge k\}; $$
thus, $\supp(\alp)=A_1\supseteq A_2\supseteq\dotsb$. Observing that $\alp$
factors as $\alp=A_1\dotsb A_{h_1}$, by Lemma~\refl{ben} we get
  $$ |\Su| \ge |\Su[A_1]|\longp|\Su[A_{h_1}]|-h_1+1. $$
Since each of $A_1\longc A_{h_1}$ has all its elements pairwise distinct,
Lemma~\refl{hui} applies to give $|\Su[A_i]|\ge\frac16|A_i|^2+1$,
$i\in[1,h_1]$. Hence,
 $$ |\Su| > \frac16\,(|A_1|^2\longp|A_{h_1}|^2)
            \ge \frac1{6h_1} (|A_1|\longp|A_{h_1}|)^2 = \frac{n^2}{6h_1}, $$
which proves~\refe{lbii}.

To prove~\refe{lbiii} we notice that if $h_2=h_1$, then the first inequality
of~\refe{lbiii} is identical to the first inequality of~\refe{lbii}, while
otherwise there is a unique element $a\in\supp(\alp)$ of multiplicity $h_1$,
and we have $A_{h_2+1}\longe A_{h_1}=\{a\}$. Therefore, in the latter case,
arguing as above,
\begin{align*}
  |\Su| &> \frac16\,(|A_1|^2\longp|A_{h_2}|^2) \\
        &\ge \frac1{6h_2}\,(|A_1|\longp|A_{h_2}|)^2  \\
        &= \frac1{6h_2}\,(n-h_1+h_2)^2;
\end{align*}
this proves~\refe{lbiii}.

We now assume that $\alp$ is stable and $n\ge(4C)^3$, and prove~\refe{newb}
and~\refe{bounds}. We notice that, in view of~\refe{lbiii}, if~\refe{newb}
holds true, then so does~\refe{bounds}, and we thus assume that~\refe{newb}
fails to hold; that is, $h_2\ge C$.


By Corollaries~\refc{TBL} and~\refc{PiPj} we have
\begin{align}
  |\Su|-2n   &\ge (h_1-2)(h_2-2) - 2 \label{e:hh_0}
  \intertext{and}
  (h_1+1)h_2 &\le |\Su| < Cn. \label{e:hh_2}
\end{align}
Combining \refe{lbii} and~\refe{hh_2} we obtain $h_2<6C^2$, and then
  $$ h_1>n-\sqrt{36C^3n}>\frac14\,n $$
by~\refe{lbiii} and the assumption $n\ge(4C)^3$. Reusing~\refe{hh_2}, we now
get $h_2<4C$. On the other hand, from~\refe{hh_0} and~\refe{lbiii},
\begin{equation}\label{e:night}
  |\Su| - 2n \ge (n-\sqrt{6Cnh_2}+h_2-2)(h_2-2) - 2.
\end{equation}
Considering $h_2$ as a variable ranging from $C$ to $4C$, the right-hand side
is a concave function of $h_2$ attaining at $h_2=C$ the value
\begin{align*}
  (n-C&\sqrt{6n} + C - 2)(C-2)-2 \\
        &\ge Cn - 2n - C(C-2)\sqrt{6n} -2 \\
        &>  Cn-C^2\sqrt{6n} - 2n, \\
        &> |\Su| - 2n
\end{align*}
and at $h_2=4C$ the value
\begin{align*}
  (n-2C&\sqrt{6n}+4C-2)(4C-2) - 2 \\
        &\ge 4Cn-8C^2\sqrt{6n} -2n \\
        &> Cn - 2n \\
        &> |\Su| - 2n.
\end{align*}
This shows that the right-hand side of~\refe{night} is larger than the
left-hand side for all $C\le h_2\le 4C$, a contradiction completing the proof
of the lemma.
\end{proof}

We are eventually ready to prove Theorem~\reft{main}.

Suppose that $C\ge 2$ is an integer, and $\alp$ is a zero-sum-free sequence
of length $n:=|\alp|\ge(4C)^3$ satisfying $|\Su|<Cn-C^2\sqrt{6n}$. We want to
show that there are an integer $n-C\sqrt{6n}<h<n$, an element
$a\in\supp(\alp)$ of order $\ord(a)>h+1$, and a set $X\seq\Su[\supp(\alp)]$
with $0\in X$ and $|X|\le C-1$, such that $\Su[\supp(\alp)]+P\seq\Suc$ and
$\Su[\supp(\alp)]\seq X+P-P$, where $P=\{a,2a\longc ha\}$.

Consider a stable sequence $\alp'$ obtained from $\alp$ by a series of
transfer operations. By Lemma~\refl{swapprop} we have $\Suc[\alp']\seq\Suc$,
$\supp(\alp')=\supp(\alp)$, and $n':=|\alp'|\ge n$; consequently, $\alp'$ is
zero-sum-free, and
  $$ |\Su[\alp']| \le |\Su| < Cn-C^2\sqrt{6n} \le Cn'-C^2\sqrt{6n'}. $$
Notice also that $n'\ge n\ge(4C)^3$. Let $h':=\h(\alp')$ be the largest
multiplicity of a term of $\alp'$. We define $h:=h'-1$ and choose $a$ to be a
term of $\alp'$ of multiplicity $\nu_{\alp'}(a)=h'$. Thus,
 $h>n'-C\sqrt{6n'}\ge n-C\sqrt{6n}$ by~\refe{bounds}, and $\ord(a)>h'=h+1$
since $\alp'$ is zero-sum-free.

The first assertion of the theorem is now immediate by observing that $\alp'$
contains $a^h\cdot\supp(\alp')$ as a subsequence, whence
  $$ \Su[\supp(\alp)]+P = \Su[\supp(\alp')]+P \seq \Suc[\alp']\seq\Suc. $$

To prove the second assertion we essentially apply the \emph{Ruzsa covering
lemma} (\refb{ru}, see also~\cite[Lemma~2.14]{b:tv}). Let $X$ be a maximal
subset of $\Su[\supp(\alp)]=\Su[\supp(\alp')]$ such that $0\in X$ and the
translates $x+P$ with $x\in X$ are pairwise disjoint. Then
  $$ |X||P| = |X+P| \le |\Su[\alp']| \le |\Su| < Cn-C^2\sqrt{6n}; $$
in view of $|P|=h>n-C\sqrt{6n}$, this gives $|X|\le C-1$.

On the other hand, by maximality of $X$, for any $g\in\Su[\supp(\alp)]$ there
exists $x\in X$ such that $(x+P)\cap(g+P)\ne\est$; that is, $g\in x+P-P$,
showing that $\Su[\supp(\alp)]\seq X+P-P$.

This completes the proof of Theorem~\reft{main}.


\bigskip

\end{document}